\documentclass{article}

\usepackage{enumitem, verbatim, graphicx}
\usepackage{amsmath, amssymb, amsthm, float, subcaption, hyperref, tikz}

\newtheorem{theorem}{Theorem}[section]
\newtheorem{proposition}[theorem]{Proposition}
\newtheorem{lemma}[theorem]{Lemma}

\newtheorem{problem}[theorem]{Problem}

\newcommand{\polytope}{\mathcal{P}}
\newcommand{\position}{{\overline{\polytope}}}
\newcommand{\randomtrace}{T_\polytope^{\mathrm{rand}}}
\newcommand{\tran}{t}

\newcommand{\set}[1]{\left\{#1\right\}}

\newcommand{\RR}{\mathbb{R}}
\newcommand{\ve}{\varepsilon}

\newcommand{\case}[1]{
\vspace{3mm}

\noindent
\textbf{Case #1:}
}

\DeclareMathOperator{\closure}{cl}

\title{On traces of randomly rolling polytopes}
\author{Kenneth Moore\thanks{R\'enyi Institute, 1053 Budapest, Re\'altanoda u. 13-15, Hungary. Supported by ERC Advanced Grants 882971``GeoScape'' and ``ERCiD.'' Email:
{\tt moore.kenneth@renyi.hu}.} \and J\'anos Pach\thanks{R\'enyi Institute, 1053 Budapest, Re\'altanoda u. 13-15, Hungary. Supported by ERC Advanced Grant 882971``GeoScape.'' Email: {\tt pach@renyi.hu}.}}
\date{}

\begin{document}

\maketitle

\begin{abstract}
    Let $\mathcal{P}$ be a three-dimensional convex polytope resting with one of its faces on the plane. At each step, $\mathcal{P}$ is allowed to roll over a randomly selected edge of the face currently lying on the plane, until the adjacent face comes to rest on the plane. The \emph{trace} of $\mathcal{P}$ is the set of all points of the plane that can be reached by a vertex of $\mathcal{P}$, starting from a fixed initial position and performing a finite sequence of rolls. We prove that if the trace of $\mathcal{P}$ has a convergent subsequence, then, with probability one, the set of points reached by the vertices of a randomly rolling copy of $\mathcal{P}$ is everywhere dense in the plane. This settles a conjecture of Hegyv\'ari.
\end{abstract}

\section{Introduction}

Consider a 3-dimensional convex polytope $\polytope$ standing with one face on a plane. It can be \emph{rolled} from its original position around one of its edges lying on the plane, until another face hits the plane. The \emph{trace} of $\polytope$, denoted $T_\polytope$, is the set of all points in the plane that coincide with a vertex of $\polytope$, after performing some sequence of rolls. If $\polytope$ is a unit cube, its trace is a square lattice.

The study of traces was initiated by Hegyv\'ari~\cite{hegyvari1995}. He characterized the density of traces for regular polyhedra and rectangular parallelepipeds, and gave a sufficient condition for a general polyhedron to have dense trace. Hegyv\'ari and Wintsche later studied the Archimedean polyhedra and showed that all but one of them have dense trace~\cite{Hegyvari-Wintsche1998}.  

We may also consider a random rolling process. At each step, an edge of the current supporting face is chosen uniformly at random, and the polytope is rolled over that edge. We denote by $\randomtrace$ the random set of points in the plane that coincide with a vertex of $\polytope$ at some time along the resulting infinite sequence of rolls, including time zero.

A set $T$ is said to be \emph{locally dense in ${\RR}^2$} if it is dense in a nonempty neighborhood of \emph{some} point $p\in\RR^2$. It is \emph{dense in} a simple continuous \emph{curve} $\gamma$ if every nonempty subcurve of $\gamma$ contains a point in $T$.

Hegyv\'ari \cite{Hegyvari2016} posed the following two problems about traces of polytopes.

\begin{problem}
    \label{prob:question1}
    Is it true that if $T_\polytope$ is dense on a smooth curve $\gamma$, then $T_\polytope$ is dense in the plane?
\end{problem}

\begin{problem} 
    \label{prob:question2}
    Assume that $T_\polytope$ is dense in the plane. Is it true that, with probability 1, $\randomtrace$ is dense in the plane?
\end{problem}

In fact, an affirmative answer to Problem~\ref{prob:question1} can be deduced from earlier work of Bicchi, Chitour, and Marigo~\cite[Theorem 1]{BicchiChitourMarigo2004}, where the rolling process was interpreted as a discrete nonholonomic system. To make our paper self-contained, in Section~\ref{sec1} we give a short direct proof of the following stronger statement.

\begin{theorem}
    \label{thm:question3}
    If a polytope $\polytope$ satisfies
    $$\inf\set{|p-q|: p,q\in T_\polytope,\ p\neq q}=0,$$ then $T_\polytope$ is dense in the plane.
\end{theorem}

The conditions of Theorem~\ref{thm:question3} merely say that the trace contains distinct points at arbitrarily small distance. In particular, it holds whenever the trace is dense in a nontrivial curve. The proof will use

\begin{lemma} [Hegyv\'ari~\cite{Hegyvari2016}]
\label{lem:local}
If the trace of a polytope $\polytope$ is locally dense in the plane, then it is also dense in the plane.
\end{lemma}

The main result of this paper provides an affirmative answer to Problem~\ref{prob:question2} for a substantially broader class of random rolling processes. For a face $F$ of $\polytope$, $\operatorname{adj}(F)$ denotes the set of faces sharing one edge with $F$, and $\deg(F)=|\operatorname{adj}(F)|.$ We call a function $\tran(F,F')$ on pairs of faces of $\polytope$ a \emph{reversible transition probability} if it satisfies the following conditions:
\begin{enumerate}
\item $\tran(F,F')>0$ whenever $F'\in \operatorname{adj}(F)$; that is, $F'$ shares one edge with $F$;
\item $\tran(F,F')=0$ whenever $F\neq F'$ and $F'\not \in \operatorname{adj}(F)$;
\item $\tran(F,F)=1-\sum_{F'\in \operatorname{adj}(F) }\tran(F,F')\geq 0$ for each face $F$;
\item There is a function $\pi$ assigning a positive weight to each face so that for all faces $F,F'$,
$$\pi(F)\tran(F,F')=\pi(F')\tran(F',F).$$
\end{enumerate}
These conditions allow a walk on the face-adjacency graph of $\polytope$ with transition probabilities $\tran(F,F')$ to form a reversible finite-state Markov Chain. The language of Markov chains will be important for the proof, and some background on them along with \emph{detailed balance} (which is condition 4 for $\tran$) may be found in~\cite[Chapter 1]{LevinPeres2017}. We can now state 

\begin{theorem}
    \label{thm:randomtrace}
Suppose the trace of a polytope $T_{\mathcal P}$ is dense in $\mathbb R^2$. 
At each step, roll $\mathcal P$ from its current supporting face $F$ to a new face $F'$ according to a reversible transition probability function $\tran(F,F')$. Then, with probability one, the set of vertices coinciding with a vertex of $\polytope$ at some time is dense in $\RR^2$.
\end{theorem}
Note that for the purposes of Problem~\ref{prob:question2}, we would set $\tran(F,F')=\frac{1}{\deg(F)}$ for $F'\in \operatorname{adj}(F)$, and $\pi(F)=\deg(F)$. An outline of the proof of Theorem~\ref{thm:randomtrace} is given at the start of Section~\ref{sec:random-outline}, and the full proof is in Sections~\ref{sec:recurrence}--\ref{sec:random-proof}.

The main ingredient of the proof is a Euclidean motion-group analogue of classical planar recurrence (Proposition~\ref{prop:recurrence}), which began with the Chung--Fuchs theorem~\cite{ChungFuchs1951}. Random walks on Euclidean motion groups have also been studied by Cr\'epel~\cite{Crepel1974}, Baldi, Bougerol, and Cr\'epel~\cite{BaldiBougerolCrepel1978} which is particularly relevant; we mention more details on this at the start of Section~\ref{sec:random-outline}.

For more recent work on products of random Euclidean isometries, see Lindenstrauss and Varj\'u~\cite{LindenstraussVarju2016}, and for some closely related questions, see \cite{dekker1959} and~\cite{wagon1993} (chap. 5, p. 69). For algorithmic reachability issues in the special case where the polyhedron is required to roll on a periodic tessellation of the plane, see  Baes et al.~\cite{BaesEtAl2022}.

\section{Proof of Theorem~\ref{thm:question3}}\label{sec1}

According to Lemma~\ref{lem:local}, it is sufficient to prove that $T_\polytope$ is locally dense. Let $\position_0$ denote the starting position of the polytope. Any edge $e$ of the face the polytope sits on can be associated with a roll $R$.  Let $R(\position_0)$ denote the polytope's new position after rolling over the given edge $e$. Note that $R(R(\position))=\position$, so $R=R^{-1}$. We write $\position(v)$ for the position of the vertex $v$ when $\polytope$ is at position $\position$. The shorthand for the position after a sequence of rolls 
\begin{equation*}
    \begin{aligned}
    \mathcal{R}(\position):=R_n(\ldots R_2(R_1(\position))),\qquad \mathcal{R}^{-1}(\position)=R_1(R_2(\ldots R_n(\position)))
    \end{aligned}
\end{equation*}
will be used. A position $\position$ of $\polytope$ with a face in the plane is \emph{valid} if there is a sequence such that $\mathcal{R}(\position_0)=\position$, and we assume, unless otherwise stated, that all positions discussed are valid. For any vertex $v$ of $\polytope$, denote by $T_\polytope(v)$ the set of all points $p$ such that there is a valid position $\position(v)$, where $\position(v)=p$.

\begin{lemma}\label{lem:samevertex}
For some vertex $v$, $T_\polytope(v)$ contains an accumulation point.
\end{lemma}

\begin{proof}
    Consider a sequence of pairs of points 
    $\set{(p_i,q_i)}_{i=1}^\infty\subseteq T_\polytope$ which satisfy $p_i\neq q_i$ and $|p_i-q_i|\to 0$. As there are finitely many pairs of vertices, we can choose a subsequence of these pairs so that $\set{p_i',q_i'}\subseteq T_\polytope(u)\times T_\polytope(v)$ for two not necessarily distinct vertices $u,v$ of $\polytope$. 
    
    Then there are valid positions $\position_{1,i}$, $\position_{2,i}$, where $p_i'$ coincides with $\position_{1,i}(u)$ and $q_i'$ coincides with $\position_{2,i}(v)$. Thus, there are sequences of rolls $\mathcal{R}_{1,i},\mathcal{R}_{2,i}$ such that $\mathcal{R}_{1,i} (\position_0) = \position_{1,i}$ and $\mathcal{R}_{2,i}(\position_0) = \position_{2,i}$. Then for $\mathcal{R}_i:=\mathcal{R}_{1,i}\circ\mathcal{R}_{2,i}^{-1}$, we see $\mathcal{R}_i(\position_{2,i})=\position_{1,i}$.

    Now choose any point $x\in T_\polytope(u)$, and a position $\position$ where $\position(u)=x$. The sequence $y_i:=(\mathcal{R} _i(\position))(u)$ satisfies $|x-y_i|=|p_i-q_i|\to 0$. In particular, every point of $T_\polytope(u)$ is the limit of a sequence in $T_\polytope(v)$. This argument can be applied symmetrically, so the converse is also true.
    
    Fix $\ve>0$ and $x\in T_\polytope(u)$. By the preceding paragraph, there is a position $\position$ with $\position(u)=x$, and a sequence of rolls $\mathcal{R}_x$, such that
    $$
    y:=(\mathcal{R}_x(\position))(v)
    $$
    satisfies
    $$
    0<|x-y|<\frac{\ve}{2}.
    $$
    Applying the same argument in the opposite direction, there is a sequence of rolls $\mathcal{R}_y$ such that
    $$
    z:=(\mathcal{R}_y(\mathcal{R}_x(\position)))(u)
    $$
    satisfies
    $$
    |z-y|<\frac12|x-y|.
    $$
    Consequently,
    $$
    0<|x-z|\leq |x-y|+|y-z|<\frac32|x-y|<\ve.
    $$
    Thus, $x$ is an accumulation point of $T_\polytope(u)$.
\end{proof}

    In what follows, let $T_v:=T_\polytope(v)$, where $v$ is the vertex given by Lemma~\ref{lem:samevertex}. The set $T_v$ has a strong self-similarity property: for every pair of points $p,q\in T_v$, there are two valid positions where $\position_1(v)=p$ and $\position_2(v)=q$. We may assume the face in the plane is the same for these two positions, otherwise, keeping $v$ fixed, roll from $\position_2$, through a sequence of edges of $\polytope$ adjacent to $v$, until a position in which the face in the plane is the same as that of $\position_1$ is reached. Now there is a rigid motion $f$ of the plane where $f(\position_1)=\position_2$. 
    
    As in the proof of Lemma~\ref{lem:samevertex}, we can also construct a sequence of rolls $\mathcal{R}$ such that $\mathcal{R}(\position_1)=\position_2$. So the valid positions of $\polytope$ are identical before and after applying $f$, hence $f(T_v)=T_v$. 
    The following lemma will help understand the structure of $T_v$, given this self-similarity. We write $\mathcal{E}(2)=\operatorname{Isom}(\RR^2)$ for the group of rigid motions of the plane.

\begin{lemma}
    \label{lem:classification}
    Let $S\subset \RR^2$ contain an accumulation point, and suppose that for every pair of points $p,q\in S$ there is a rigid motion $g_{p,q}$ where $g_{p,q}(p)=q$, and $g_{p,q}(S)=S$. Then $\closure S$ is either 
    \begin{enumerate}
        \item a union of parallel lines (possibly just one line), or
        \item a circle, or
        \item all of $\RR^2$.
    \end{enumerate}
\end{lemma}

\begin{proof}
    Let $G$ be the group generated by all the rigid motions $g_{p,q}$, and let $H=\closure G$ within $\mathcal{E}(2)$. For any given $x\in S$, we claim that 
    $$\closure S = H\circ x:=\set{h(x):h\in H}.$$ 
    Note that $S=G\circ x$.

    First, we see that $H\circ x$ is closed. Indeed, let $h_n$ be a sequence of rigid motions in $H$ such that $h_n(x) \to y$. These maps have the form $A_nx+b_n$ where $b_n\in \RR^2$ and $A_n$ is an orthogonal matrix in $\mathcal{O}(2)$. After passing to a subsequence, the $A_n$ converge to $A\in \mathcal{O}(2)$, and then $h_n(x)-A_nx=b_n$ converges as well. Hence, the subsequence converges to some $h\in H$ where $y=h(x)\in H\circ x$. 
    
    Since $S=G\circ x\subseteq H\circ x$ which is closed, $\closure S\subseteq H\circ x$. On the other hand, any $h\in H$ is a limit of elements in $G$, so 
    $h(x)\in \closure (G\circ x)=\closure S.$ 
    Let $H_0$ be the identity component of $H$. The closed connected subgroups of $\mathcal{E}(2)$ consist of the following items. $H_0$ could be
    \begin{enumerate}[label=(\alph*)]
        \item the trivial group $\set{e}$,
        \item a set of translations along one direction,
        \item a set of rotations around a fixed point,
        \item all translations, or
        \item all orientation-preserving rigid motions of the plane.
    \end{enumerate}
    Case (a) is not possible as $S$ contains an accumulation point, so $H$ has dimension at least 1. If cases (b) or (c) held, they would imply that $\closure S$ consists of a union of parallel lines or a circle respectively, and both cases (d) and (e) would imply that $\closure S = \RR^2$.
\end{proof}


In light of Lemma~\ref{lem:classification}, there are three cases for what $T_v$ may look like.
\case{1} The closure of $T_v$ is a circle $\gamma$; the rigid motions sending faces to faces are thus rotations about the center of the circle. When rolling the polytope around a given point $p\in T_v$, we necessarily find a second instance of the same face on the plane connected to $p$, and thus a distinct circular arc through $p$. This contradicts that the closure of $T_v$ is a circle.

    \begin{figure}[ht]
        \centering
        \includegraphics[width=0.95\linewidth]{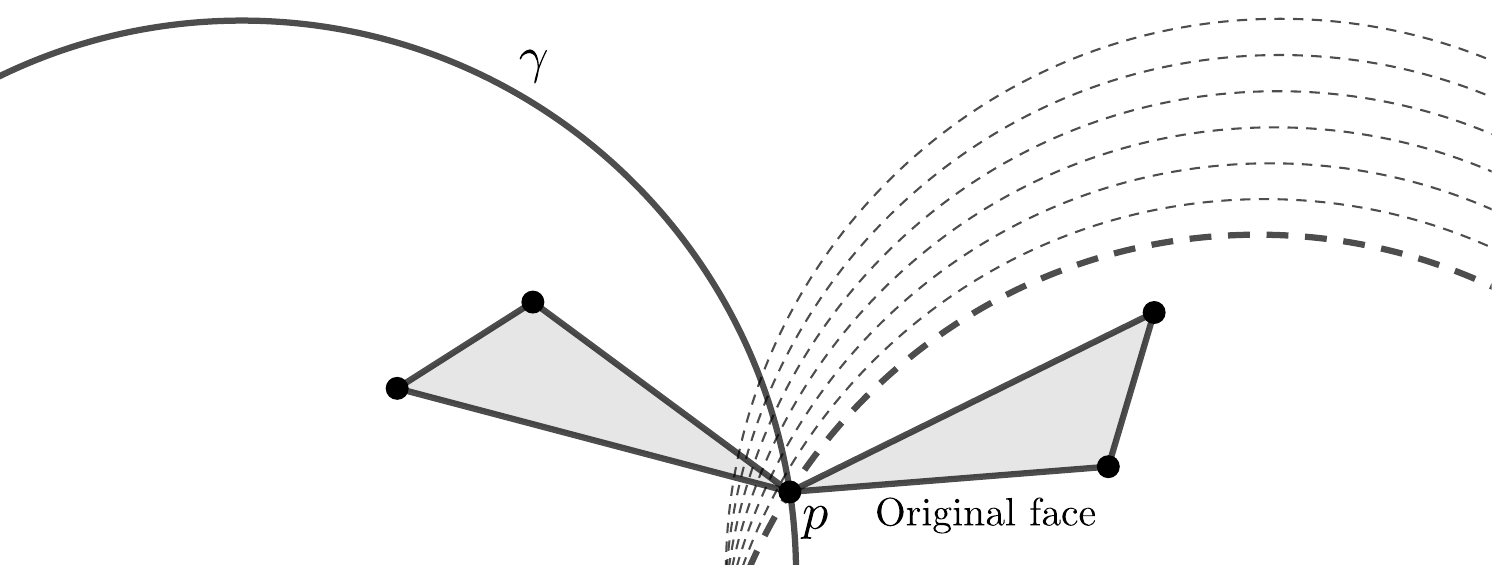}
        \caption{An area swept out by copies of $\gamma$}
        \label{fig:case1}
    \end{figure}
    
\case{2} $T_v$ is a line or a union of lines $\gamma$; the rigid motions sending faces to faces are thus parallel translations or 180 degree rotations. 
If, when rolled around a given point $p\in T_v$, we find a second face that is not at $180^{\circ}$ from the first one, we are done similarly to the previous case. This is because there is then a second line through $p$ skew to $\gamma$.

    \begin{figure}[ht]
        \centering
        \includegraphics[width=0.65\linewidth]{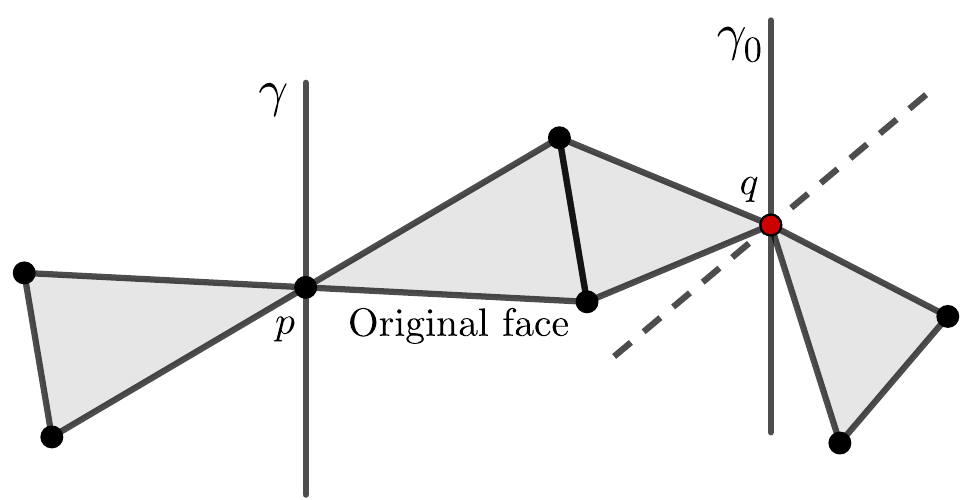}
        \caption{Finding another area swept out by copies of $\gamma$}
        \label{fig:case2}
    \end{figure}
    
Otherwise, the angles of the faces at $p$ must sum to $180^{\circ}$ exactly. In this case, choose a vertex $q$ which does not have a sum $180^{\circ}$ from its adjacent faces (there must be such a vertex, otherwise, $\polytope$ is a tetrahedron, whose trace is the triangular lattice and has no accumulation points). From every parallel face generating $T_v$, follow a common sequence of rolls until $q$ lies in the plane as in Figure~\ref{fig:case2}. Repeat the above argument on the set of all such $q$, which is dense in a translate $\gamma_0$ of $\gamma$.

\case{3} The closure of $T_v$ is $\RR^2$. Since $T_v\subseteq T_\polytope$, we are done.

\section{Proof of Theorem~\ref{thm:randomtrace}}
\label{sec:random-outline}

We first give a rough outline of the argument. In short, the proof works by fixing the initial supporting face $F_0$, and observing the random rolling process only at times when it returns to $F_0$. This becomes a random walk on a subgroup of the planar rigid motions, which is easier to analyze. 
For simplicity, we continue to denote the random trace from Theorem~\ref{thm:randomtrace} by $\randomtrace$, although the transition function may differ from the uniform one considered by Hegyv\'ari.
\vspace{2mm}

We first show a recurrence statement (Proposition~\ref{prop:recurrence}) for planar isometry walks. It states that a symmetric reversible sequence of random rigid motions with finite second moment is recurrent, in the sense that the $k^{\text{th}}$ displacement map $g_k$ is within each neighborhood of the identity map infinitely often.

To prove this, we show the existence of a suitable point $c\in\RR^2$, where the points $g_k(c)$ form a planar martingale (Lemma~\ref{lem:process_centre}) whose typical displacement is of order $\sqrt{k}$ (Lemma~\ref{lem:half-mass}). 

Next we prove that the probabilities of $g_k$ returning to a neighborhood of the identity, summed over $k$ diverge (Lemma~\ref{lem:divergent-green}). To do so, note that a ball of radius order $\sqrt{k}$ in $\mathcal{E}(2)$ can be covered by $O(k)$ translates of a fixed small neighborhood. Two independent copies of the walk with then land in the same neighborhood almost surely, and we leverage this to get a return probability of order at least $1/k$.

The strong Markov property turns the resulting divergent sum into almost-sure recurrence. This will conclude the proof of Proposition~\ref{prop:recurrence}, the main technical difficulty in the proof.
\vspace{3mm}

We now consider a process equivalent to our rolling procedure. At the $k^\text{th}$ time $\polytope$ returns to face $F_0$, the new placement of $F_0$ differs from its initial placement by a rigid motion $g_k$. The portions of the face process between successive returns to $F_0$ are independent and identically distributed. That is,
$g_k=h_1h_2\cdots h_k,$
where the $h_i$ are independent and identically distributed random rigid motions.

We then argue that the proposition applies here, starting by observing that the law of $h_1$ is symmetric. That is, every first-return sequence of faces and the same sequence followed in reverse have equal probabilities, while their rigid motions are inverses of one another. Moreover, the probability of a return excursion becomes exponentially smaller with its length. Since a single roll moves points on the polytope by a bounded distance, it follows that $\mathbb E|h_1(0)|^2<\infty,$ or, the second moment is finite.

Finally, fix any open disk $B$. Since $T_\polytope$ is dense, there is a finite rolling sequence $\mathcal{R}$ from $\position_0$ which places some vertex at a point $x\in B$. Whenever $g_k$ is sufficiently close to the identity, following $\mathcal{R}$ from the position at the $k^{\text{th}}$ return will place that vertex at $g_k(x)\in B$.
There are almost surely infinitely many such return times, and each gives a new trial with the same positive probability of following $\mathcal{R}$. Hence, one of the trials succeeds almost surely. Applying this to a countable basis of disks proves that $\randomtrace$ is almost surely dense.

\subsection{A recurrence result for planar isometry walks}
\label{sec:recurrence}

We first establish the recurrence statement needed for the random rolling argument. Recall that $\mathcal{E}(2)$ denotes the group of planar Euclidean isometries, and $\mathcal{O}(2)$ is the group of orthogonal matrices. This subsection is entirely dedicated to proving the following proposition, which says that any random walk in a closed subgroup of the isometries which is symmetric and has finite second moment is recurrent.

We note that Euclidean isometry recurrence has been well studied and for example, Corollary 1 in \cite{BaldiBougerolCrepel1978} is nearly the recurrence statement we needed. The differences are that they do not require symmetry, and ours applies to any closed subgroup. 

\begin{proposition}
\label{prop:recurrence}
Let $G$ be a closed subgroup of $\mathcal{E}(2)$, and let
$h_1,h_2,\ldots$ be independent and identically distributed $G$-valued random
variables. Assume that
$$
h_1\overset{d}=h_1^{-1}
\qquad\text{and}\qquad
\mathbb E\lvert h_1(0)\rvert^2<\infty.
$$
Define
$$
g_n=h_1h_2\cdots h_n,
\qquad
g_0=\operatorname{id}.
$$
Then, for every open neighborhood $U$ of the identity in $G$,
$$
\mathbb P(g_n\in U\text{ infinitely often})=1.
$$
\end{proposition}
We first establish a few lemmas. Write each $g\in \mathcal{E}(2)$ as $g(z)=A_gz+b_g,$ where $A_g\in \mathcal{O}(2)$ and $b_g=g(0)\in\RR^2$. 
The first lemma gives a special point, where the expected shift of that point under $h_1$ is the point itself.

\begin{lemma}
    \label{lem:process_centre}
    Under the conditions of Proposition~\ref{prop:recurrence}, there is a point $c\in\RR^2$ satisfying
\begin{equation}
\mathbb E[h_1(c)]=c.
\label{eq:center}
\end{equation}
\end{lemma}
\begin{proof}
Put
$$
\overline A=\mathbb E[A_{h_1}],
\qquad
\overline b=\mathbb E[b_{h_1}].
$$
Then equation~\eqref{eq:center} is equivalent to
$$
(I-\overline A)c=\overline b,
$$
which is solvable if and only if
\begin{equation}
\label{eq:perpendicular}
\overline b\perp\ker(I-\overline A^{T}).
\end{equation}
Take $v\in\ker(I-\overline A^{T})$. Then $\overline A^{T}v=v$, and hence
$$
\lvert v\rvert^2
 =\langle\overline A^{T}v,v\rangle
 =\mathbb E\langle A_{h_1}^{T}v,v\rangle
 \leq \lvert v\rvert^2.
$$
The integrand is at most $\lvert v\rvert^2$, so equality of the expectations forces
$\langle A_{h_1}^{T}v,v\rangle=\lvert v\rvert^2$ almost surely. Since
$A_{h_1}$ is orthogonal, equality in Cauchy--Schwarz gives $A_{h_1}^{T}v=v$ almost surely, and, therefore, also $A_{h_1}v=v$ almost surely.

For every $g\in \mathcal{E}(2)$,
$$
b_{g^{-1}}=-A_g^{T}b_g.
$$
Using the symmetry $h_1\overset d=h_1^{-1}$, we obtain
$$
\begin{aligned}
\langle\overline b,v\rangle
 &=\mathbb E\langle b_{h_1},v\rangle
  =\mathbb E\langle b_{h_1^{-1}},v\rangle \\
 &=\mathbb E\langle-A_{h_1}^{T}b_{h_1},v\rangle
  =-\mathbb E\langle b_{h_1},A_{h_1}v\rangle \\
 &=-\mathbb E\langle b_{h_1},v\rangle
  =-\langle\overline b,v\rangle.
\end{aligned}
$$
Thus, $\langle\overline b,v\rangle=0$, which proves \eqref{eq:perpendicular} and, therefore, the existence of a vector $c\in\RR^2$ satisfying \eqref{eq:center}.
\end{proof}

Let $\lVert\cdot\rVert$ denote the Frobenius norm on $2\times2$ matrices, and define for any two rigid motions $g,h\in \mathcal{E}(2)$,
$$
\rho(g,h)=\lvert g(0)-h(0)\rvert+\lVert A_g-A_h\rVert.
$$
This metric induces the usual topology on $\mathcal{E}(2)$. It is left-invariant: for
$k,g,h\in \mathcal{E}(2)$, orthogonality of $A_k$ gives
$$
\rho(kg,kh)
 =\lvert A_k(g(0)-h(0))\rvert
  +\lVert A_k(A_g-A_h)\rVert
 =\rho(g,h).
$$
The next lemma will control how far from the identity the rigid motions $g_n$ are. 

\begin{lemma}
\label{lem:half-mass}
    Under the conditions of Proposition~\ref{prop:recurrence}, there is a constant $C$ such that, for every $n\geq1$,
    \begin{equation}
        \label{eq:half-mass}
        \mathbb P\bigl(\rho(g_n,\operatorname{id})\leq C\sqrt n\bigr)\geq\frac12.
    \end{equation}
\end{lemma}
\begin{proof}
Let $c$ be the point given by Lemma~\ref{lem:process_centre}, and let $\mathcal F_n$ be the sigma algebra generated by $(h_1,\ldots,h_n)$ and define
$$
M_n=g_n(c).
$$
Since $g_n$ is affine and $h_{n+1}$ is independent of $\mathcal F_n$,
$$
\mathbb E[M_{n+1}\mid\mathcal F_n]
 =g_n\bigl(\mathbb E[h_{n+1}(c)]\bigr)
 =g_n(c)
 =M_n.
$$
Thus, $M_n$ is a square-integrable $\RR^2$-valued martingale. Moreover,
$$
M_{n+1}-M_n=g_n(h_{n+1}(c))-g_n(c),
$$
so because $g_n$ is an isometry,
$$
\lvert M_{n+1}-M_n\rvert=\lvert h_{n+1}(c)-c\rvert.
$$

Set
$$
\sigma^2 = \mathbb E\lvert h_1(c)-c\rvert^2<\infty,
$$
where finiteness follows from
$\lvert h_1(c)-c\rvert\leq\lvert h_1(0)\rvert+2\lvert c\rvert$.
The martingale differences are orthogonal in $L^2$ and, therefore,
\begin{equation}
\mathbb E\lvert M_n-c\rvert^2=n\sigma^2.
\label{eq:martingale-second-moment}
\end{equation}
Also, again using that $g_n$ is an isometry,
$$
\begin{aligned}
\lvert g_n(0)\rvert
 &\leq \lvert g_n(0)-g_n(c)\rvert
       +\lvert g_n(c)-c\rvert+\lvert c\rvert \\
 &=2\lvert c\rvert+\lvert M_n-c\rvert.
\end{aligned}
$$
It follows from \eqref{eq:martingale-second-moment} that
\begin{equation}
\mathbb E\lvert g_n(0)\rvert^2=O(n).
\label{eq:translation-second-moment}
\end{equation}
Since $\mathcal{O}(2)$ is bounded, \eqref{eq:translation-second-moment} implies
\begin{equation}
\mathbb E\bigl[\rho(g_n,\operatorname{id})^2\bigr]=O(n).
\label{eq:rho-second-moment}
\end{equation}
The claim in the lemma now follows from~\eqref{eq:rho-second-moment} and Markov's inequality.
\end{proof}

We will now use the preceding lemma to show that the probabilities of $g_n$ coming arbitrarily close to the identity map diverge when summed over $n$. We use the term \emph{identity neighborhood} to mean an open neighborhood of the identity map $\operatorname{id}$ in $G$.

\begin{lemma}
    \label{lem:divergent-green}
    Let $g_n$ and $G$ be defined as in Proposition~\ref{prop:recurrence}. Fix an identity neighborhood $U\subseteq G$. Then
    \begin{equation}
    \sum_{n=0}^{\infty}\mathbb P(g_n\in U)=\infty.
    \label{eq:divergent-green}
    \end{equation}
\end{lemma}

\begin{proof}
Choose an identity neighborhood
$V\subseteq G$ such that $V^{-1}V\subseteq U.$
Let
$$
B_n=\{g\in G:\rho(g,\operatorname{id})\leq C\sqrt n\},
$$
where $C$ is the constant given by Lemma~\ref{lem:half-mass}. We claim that $B_n$ can be covered by at most $Kn$ left translates of $V$, where
$K$ is a constant independent of $n$. To see this, choose $\ve>0$ such that
$$
\{g\in G:\rho(g,\operatorname{id})<\ve\}\subseteq V.
$$
Under the correspondence $g\leftrightarrow(b_g,A_g)\in\RR^2\times \mathcal{O}(2)$, the set $B_n$ has $b_g$ in a disk of radius $C\sqrt n$ in $\RR^2$, while $A_g$ lies in the compact space $\mathcal{O}(2)$. The translation disk can be covered by $O_{C,\ve}(n)$ sets of diameter at most $\ve/2$, and $\mathcal{O}(2)$ can be covered by $O_\ve(1)$ sets of diameter at most $\ve/2$. Their products, therefore, cover $B_n$ by $O_{C,\ve}(n)$ sets of $\rho$-diameter less than $\ve$. Choosing an element $f_j\in G$ from each nonempty such set shows that the set is contained in
$$
f_j\{g\in G:\rho(g,\operatorname{id})<\ve\}\subseteq f_jV.
$$
Consequently, for some constant $K=K(C,\ve)$,
$$
B_n\subseteq\bigcup_{j=1}^{N_n}f_jV,
\hspace{2mm}
\text{for an integer }N_n\leq Kn.
$$

Partition $B_n$ into measurable sets $C_1,\ldots,C_{N_n}$ satisfying
$C_j\subseteq f_jV$, and put
$$
p_j=\mathbb P(g_n\in C_j).
$$
By Lemma~\ref{lem:half-mass},
$$
\sum_{j=1}^{N_n}p_j\geq\frac12.
$$
Let $g_n'$ be an independent copy of $g_n$. If $g_n$ and $g_n'$ lie in the same set $C_j$, then
$$
g_n^{-1}g_n'\in V^{-1}V\subseteq U.
$$
Hence, by Cauchy--Schwarz,
$$
\begin{aligned}
\mathbb P(g_n^{-1}g_n'\in U)
 &\geq\sum_{j=1}^{N_n}p_j^2
 \geq\frac{\left(\sum_jp_j\right)^2}{N_n}
 \geq\frac{1}{4Kn}.
\end{aligned}
$$

Now,
$$
g_n^{-1}=h_n^{-1}\cdots h_1^{-1}.
$$
Because the increments are i.i.d.\ and symmetric, $g_n^{-1}$ has the same law as
$g_n$. It is independent of $g_n'$, so $g_n^{-1}g_n'$ has the same law as $g_{2n}$. Therefore, for $K'=\frac{1}{4K}>0$,
\begin{equation}
\mathbb P(g_{2n}\in U)\geq\frac{K'}{n}.
\label{eq:return-lower-bound}
\end{equation}
Summing this equation over $n$ is divergent, and equation~\eqref{eq:divergent-green} follows.
\end{proof}

It only remains to leverage Lemma~\ref{lem:divergent-green} into almost-sure recurrence. 

\begin{proof}[Proof of Proposition~\ref{prop:recurrence}]
Suppose, toward a
contradiction, that for some identity neighborhood $U\subseteq G$,
$$
p:=\mathbb P(\exists n\geq1:g_n\in U)<1.
$$
Choose an identity neighborhood $V$ with $V^{-1}V\subseteq U$. Let
$T_1<T_2<\cdots$ be the successive visit times to $V$, with $T_1=0$, since
$g_0=\operatorname{id}\in V$. On the event $\{T_j<\infty\}$, write $g_{T_j}=x\in V$. By the strong Markov property, the future steps in the walk have the same law as $(g_n)_{n\geq0}$, independently of the past (see, for instance, \cite[Section A.3.]{LevinPeres2017} for a description of the strong Markov property). A return to $V$ requires that the future walk enters
$$
x^{-1}V\subseteq V^{-1}V\subseteq U.
$$
Thus,
$$
\mathbb P(T_{j+1}<\infty\mid\mathcal F_{T_j})\leq p
\qquad\text{on }\{T_j<\infty\},
$$
and induction gives
$$
\mathbb P(T_j<\infty)\leq p^{j-1}.
$$
Therefore, the expected total number of visits to $V$ is at most
$$
\sum_{j=1}^{\infty}p^{j-1}<\infty.
$$

On the other hand, the expected total number of visits to $V$ is precisely the expectation of the sum of indicator functions
$$
\mathbb E\left[\sum_{m=1}^\infty \chi_{\set{g_m\in V}} \right] = 
\sum_{m=1}^\infty \mathbb E[\chi_{\set{g_m\in V}}]
= \sum_{m=0}^{\infty}\mathbb P(g_m\in V),
$$
which is infinite by Lemma~\ref{lem:divergent-green} applied to $V$. This contradiction shows
that, for every identity neighborhood $U\subseteq G$,
\begin{equation}
\mathbb P(\exists n\geq1:g_n\in U)=1.
\label{eq:one-return}
\end{equation}

Finally, fix an identity neighborhood $W\subseteq G$. Starting from any point $x\in W$, a future return to $W$ is equivalent to the relative walk entering $x^{-1}W$, which is again an identity neighborhood. By \eqref{eq:one-return}, this happens with probability one. Defining the successive return times to $W$ recursively and applying the strong Markov property at each one shows that all of them are finite almost
surely. Hence,
$$
\mathbb P(g_n\in W\text{ infinitely often})=1.
$$
\end{proof}

\subsection{Applying the proposition to the rolling polytope}\label{sec:return-times}

In this section we convert our random polytope rolling process into a scheme satisfying the conditions of Proposition~\ref{prop:recurrence}. Let $F_0$ be the labeled face on which $\polytope$ initially rests. Choose a fixed point $a\in F_0$, and choose planar coordinates so that
$\position_0(a)=0$. 

Let \(X_n\) be the labeled face on which \(\polytope\) rests after \(n\) rolls. Then \((X_n)\) is
a Markov chain on the finite, connected dual graph of \(\polytope\) with transition probabilities given by $\tran$. Define the successive return times to the initial face by
$$
\tau_0=0,
\qquad
\tau_{k+1}=\inf\{n>\tau_k:X_n=F_0\}.
$$
Since the state space is finite and the chain is irreducible, every $\tau_k$ is finite
almost surely.

At time $\tau_k$, the same labeled face $F_0$ again lies in the plane. Its current placement differs from its initial placement by a unique planar isometry. Denote this isometry by
$g_k\in \mathcal{E}(2)$, with $g_0=\operatorname{id}.$ Define the relative increment of the $(k+1)$-st excursion by
$$
h_{k+1}=g_k^{-1}g_{k+1}.
$$
The face excursions between successive returns to $F_0$ are independent and identically distributed by the strong Markov property. Rolling is equivariant under planar isometries, so the relative isometry produced by an excursion is determined by its labeled face path. Consequently, $h_1,h_2,\ldots$ are independent and identically distributed.


Consider a first-return excursion
$$
\gamma=(F_0,F_1,\ldots,F_{m-1},F_m=F_0),
$$
with no intermediate visit to $F_0$. Its probability is
$$
\mathbb P(\gamma)=\prod_{i=0}^{m-1}\tran(F_i,F_{i+1})
$$
The reversed excursion is
$$
\gamma^{-1}=(F_0,F_{m-1},\ldots,F_1,F_0),
$$
and it has the probability
$$
\prod_{i=0}^{m-1}\tran(F_{i+1},F_i)=\prod_{i=0}^{m-1}\frac{\pi(F_i)}{\pi(F_{i+1})}\tran(F_i,F_{i+1})=\frac{\pi(F_0)}{\pi(F_{m})}\prod_{i=0}^{m-1}\tran(F_i,F_{i+1}),
$$
which is identical since $F_0=F_m$. Each
roll in the reversed excursion is the inverse of the corresponding roll in the original
excursion, in reverse order. Thus, if $\gamma$ produces the relative isometry $h$, then
$\gamma^{-1}$ produces $h^{-1}$. It follows that
\begin{equation}
h_1\overset d=h_1^{-1}.
\label{eq:increment-symmetry}
\end{equation}


Let $L=\tau_1$ be the length of the first return excursion. There exist constants $C$ and $r$ such that $\mathbb P(L>n)\leq Cr^n$. Indeed, because the face chain is finite and irreducible, there exist an integer $M\geq1$ and a number $p>0$ such that, from every face, there is probability at least $p$ of visiting $F_0$ within the next $M$ steps; for the starting state $F_0$, one uses a fixed positive-length return path (such as, rolling over an edge and then back). Applying the Markov property in successive blocks of length $M$ gives
\begin{equation}
\label{eq:exponential_tail}
\mathbb P(L>jM)\leq(1-p)^j,
\qquad j\geq0.
\end{equation}

There is a constant $D_\polytope<\infty$, depending only on the diameter of $\polytope$, such that one roll moves any
fixed material point of $\polytope$ by Euclidean distance at most $D_\polytope$. During an excursion of
length $L$, the chosen material point $a$, therefore, moves by total distance at most
$D_\polytope L$. At the end of the excursion this point lies in the supporting plane at the
position $h_1(0)$, and hence
$$
\lvert h_1(0)\rvert\leq D_\polytope L.
$$
Equation~\eqref{eq:exponential_tail} now yields
\begin{equation}
\mathbb E\lvert h_1(0)\rvert^2<\infty.
\label{eq:increment-second-moment}
\end{equation}

Let $\mu$ be the law of $h_1$, and define the deterministic closed subgroup
$$
G=\overline{\langle\operatorname{supp}\mu\rangle}\subseteq \mathcal{E}(2).
$$
Note that $\operatorname{supp}\mu$ is the \emph{support} of $\mu$; the set of isometries that can occur or be approximated by isometries that occur with positive probability, $\langle\operatorname{supp}\mu \rangle$ denotes the group generated by those isometries, and we then take the closure to form $G$. Every $h_k$ is $G$-valued almost surely. Let $g_k=h_1h_2\cdots h_k.$
By \eqref{eq:increment-symmetry}, \eqref{eq:increment-second-moment}, and
Proposition~\ref{prop:recurrence}, for every identity neighborhood $U\subseteq G$,
\begin{equation}
\mathbb P(g_k\in U\text{ infinitely often})=1.
\label{eq:H-recurrence}
\end{equation}

\subsection{Establishing density of the random trace} 
\label{sec:random-proof}

We have shown that the polytope's face-return isometries are arbitrarily close to the identity map infinitely often. What remains in this section is to make an argument that any point in the plane is reached almost surely. This is intuitively easy, as any open set in the plane can be entered by finite a sequence of rolls from the starting point. Since we pass by the starting point infinitely often, we have infinite attempts to run said finite sequence of rolls.

Fix a nonempty open disk $B\subseteq\RR^2$. From density, there is a finite rolling sequence $\mathcal R$ of length $m$ and a vertex of the polytope $v$ such that $\mathcal{R}(\position_0)(v)\in B$. Let $x=\mathcal{R}(\position_0)(v)$, and Define
$$
U_B=\{h\in G:h(x)\in B\}.
$$
The correspondence $h\mapsto h(x)$ is continuous, and
$\operatorname{id}(x)=x\in B$. Thus, $U_B$ is an open neighborhood of the identity in
$G$. By Proposition~\ref{prop:recurrence}, almost surely there are infinitely many $k$ such that $g_k\in U_B.$ From the definition of $U_B$, at time $\tau_k$, the vertex that originally landed at $x$ now lands at $g_k(x)\in B.$

The probability of following the prescribed sequence $\mathcal{R}$, starting from any placement
on $F_0$, is the fixed positive number $q_{\mathcal{R}}$. 
We now select the trials in a way that permits a direct use of the strong Markov property.
Let $(\mathcal G_n)$ be the filtration generated by the first $n$ rolls. Let $T_1$ be
the first return time $\tau_k$ for which $g_k\in U_B$, and, recursively, let
$$
T_{r+1}
 =\inf\{\tau_k:\tau_k\geq T_r+m\text{ and }g_k\in U_B\}.
$$
These are stopping times, and the fact that $g_k\in U_B$ infinitely often almost surely, every $T_r$ is finite almost surely. The intervals of rolls
$(T_r,T_r+m]$ are pairwise disjoint.

Let $A_r$ be the event that the $m$ rolls immediately after $T_r$ follow ${\mathcal{R}}$.
By the strong Markov property,
$$
\mathbb P(A_r\mid\mathcal G_{T_r})=q_{\mathcal{R}}
\qquad\text{almost surely}.
$$
Moreover, because the trial intervals do not overlap, the outcomes of the first $r-1$
trials are $\mathcal G_{T_r}$-measurable. Therefore, by induction,
$$
\begin{aligned}
\mathbb P(A_1^c\cap\cdots\cap A_N^c)
 &=\mathbb E\left[
   \mathbf 1_{A_1^c\cap\cdots\cap A_{N-1}^c}
   \mathbb P(A_N^c\mid\mathcal G_{T_N})
   \right] \\
 &=(1-q_{\mathcal{R}})\mathbb P(A_1^c\cap\cdots\cap A_{N-1}^c) \\
 &=(1-q_{\mathcal{R}})^N.
\end{aligned}
$$
Letting $N\to\infty$ shows that at least one of these trials succeeds almost surely. A
successful trial places a vertex in $B$, so
\begin{equation}
\mathbb P\bigl(B\cap\randomtrace\neq\varnothing\bigr)=1.
\label{eq:disk-hit}
\end{equation}

Finally, consider the countable collection of all open disks with rational centers and
positive rational radii. By \eqref{eq:disk-hit}, each such disk meets
$\randomtrace$ with probability one. Intersecting these countably many
probability-one events, we find that, almost surely, every rational disk meets
$\randomtrace$. Since these disks form a basis for the topology of
$\RR^2$, the random trace is almost surely dense. This proves
Theorem~\ref{thm:randomtrace}.

\section*{Acknowledgements}

Some details of the proof of Theorem~\ref{thm:randomtrace} were worked out with the assistance of ChatGPT Plus. 

\bibliographystyle{abbrv}
\bibliography{bibliography}

@article{dekker1959,
  title={On reflections in Euclidean spaces generating free products},
  author={Dekker, Th. J.},
  journal={Nieuw Arch. Wisk},
  volume={3},
  number={7},
  pages={57--60},
  year={1959}
}

@article{hegyvari1995,
  title={On the trace of polyhedra},
  author={Hegyv{\'a}ri, Norbert},
  journal={Geometriae Dedicata},
  volume={55},
  number={1},
  pages={71--74},
  year={1995},
  publisher={Springer}
}

@article{Hegyvari2016,
  author    = {Hegyv{\'a}ri, Norbert},
  title     = {On Deterministic and Random Rolling of Polyhedra},
  journal   = {International Electronic Journal of Geometry},
  volume    = {9},
  number    = {1},
  pages     = {85--88},
  year      = {2016},
  doi       = {10.36890/iejg.591896},
  url       = {https://doi.org/10.36890/iejg.591896}
}

@article {Hegyvari-Wintsche1998,
    AUTHOR = {Hegyv\'ari, Norbert and Wintsche, Gergely},
     TITLE = {The rolling polyhedra},
   JOURNAL = {Publ. Math. Debrecen},
  FJOURNAL = {Publicationes Mathematicae Debrecen},
    VOLUME = {52},
      YEAR = {1998},
    NUMBER = {1-2},
     PAGES = {43--54},
       DOI = {10.5486/pmd.1998.1741},
       URL = {https://doi.org/10.5486/pmd.1998.1741},
}

@book{wagon1993,
  title={The Banach-Tarski Paradox},
  author={Wagon, Stan},
  volume={24},
  year={1993},
  publisher={Cambridge University Press}
}

@article{BicchiChitourMarigo2004,
  author  = {Antonio Bicchi and Yacine Chitour and Alessia Marigo},
  title   = {Reachability and Steering of Rolling Polyhedra: A Case Study in Discrete Nonholonomy},
  journal = {IEEE Transactions on Automatic Control},
  volume  = {49},
  number  = {5},
  pages   = {710--726},
  year    = {2004},
  doi     = {10.1109/TAC.2004.826727} 
}

@inproceedings{BaesEtAl2022,
  author    = {Akira Baes and Erik D. Demaine and Martin L. Demaine and Elizabeth Hartung and Stefan Langerman and Joseph O'Rourke and Ryuhei Uehara and Yushi Uno and Aaron Williams},
  title     = {Rolling Polyhedra on Tessellations},
  booktitle = {11th International Conference on Fun with Algorithms (FUN 2022)},
  series    = {Leibniz International Proceedings in Informatics},
  volume    = {226},
  pages     = {6:1--6:16},
  publisher = {Schloss Dagstuhl--Leibniz-Zentrum f{\"u}r Informatik},
  year      = {2022},
  doi       = {10.4230/LIPIcs.FUN.2022.6}
}

@book{LevinPeres2017,
  author    = {David A. Levin and Yuval Peres},
  title     = {Markov Chains and Mixing Times},
  edition   = {2},
  note      = {With contributions by Elizabeth L. Wilmer},
  publisher = {American Mathematical Society},
  address   = {Providence, RI},
  year      = {2017},
  isbn      = {978-1-4704-2962-1}
}

@article{ChungFuchs1951,
  author  = {K. L. Chung and W. H. J. Fuchs},
  title   = {On the Distribution of Values of Sums of Random Variables},
  journal = {Memoirs of the American Mathematical Society},
  volume  = {1951},
  number  = {6},
  pages   = {1--12},
  year    = {1951}
}

@article{Crepel1974,
  author  = {Pierre Cr{\'e}pel},
  title   = {Marches al{\'e}atoires sur le groupe des d{\'e}placements du plan},
  journal = {Comptes Rendus de l'Acad{\'e}mie des Sciences de Paris, S{\'e}rie A},
  volume  = {278},
  pages   = {961--964},
  year    = {1974}
}

@article{BaldiBougerolCrepel1978,
  author  = {Paolo Baldi and Philippe Bougerol and Pierre Cr{\'e}pel},
  title   = {Th{\'e}or\`eme central limite local sur les extensions compactes de $\mathbb{R}^d$},
  journal = {Annales de l'Institut Henri Poincar{\'e}, Section B},
  volume  = {14},
  number  = {1},
  pages   = {99--111},
  year    = {1978}
}

@article{LindenstraussVarju2016,
  author  = {Elon Lindenstrauss and P{\'e}ter P. Varj{\'u}},
  title   = {Random Walks in the Group of Euclidean Isometries and Self-Similar Measures},
  journal = {Duke Mathematical Journal},
  volume  = {165},
  number  = {6},
  pages   = {1061--1127},
  year    = {2016},
  doi     = {10.1215/00127094-3167490}
}

\end{document}